\documentclass[a4paper,10pt]{article}

\usepackage[latin1]{inputenc}
\usepackage{mathrsfs,amssymb,amsmath,amsthm}
\usepackage{amsfonts}
\usepackage{amssymb}
\usepackage{amsfonts}
\usepackage{dsfont}
\usepackage{enumerate}

\usepackage[usenames,dvipsnames]{color}

\definecolor{rouge}{RGB}{255,0,0}
\definecolor{bleu}{RGB}{0,0,255}
\definecolor{orange}{RGB}{122,0,122}

\usepackage{graphicx}
\DeclareGraphicsExtensions{.jpg,.pdf,.png,.gif,.eps}

\newtheorem{theorem}{Theorem}
\newtheorem{definition}{Definition}
\newtheorem{proposition}{Proposition}

\newtheorem{lemma}{Lemma}
\newtheorem{remark}{Remark}
\newtheorem{example}{Example}

\newcounter{step}

\newenvironment{problem}[1][Problem]{\begin{trivlist}
\item[\hskip \labelsep {\bfseries #1}]}{\end{trivlist}}

\newcommand{\m}{\mathbb}

\newcommand{\1}{\mathds{1}}

\newcommand{\N}{\ensuremath{\mathds{N}}}

\renewcommand{\phi }{\varphi }

\def \1{\mathbf{1}}
\renewcommand{\P }{\mathds{P}}
\renewcommand{\geq }{\geqslant }
\renewcommand{\leq }{\leqslant }

\DeclareMathOperator*{\argmin}{argmin}
\DeclareMathOperator*{\bary}{Bar}
\DeclareMathOperator*{\sco}{Sco}
\DeclareMathOperator*{\graph}{Graph}
\DeclareMathOperator*{\diam}{diam}
\begin{document}
\bibliographystyle{plain}

\title{Pathwise uniform value in gambling houses and Partially Observable Markov Decision Processes}

\author{Xavier Venel \thanks{CES, Universit\'e Paris 1 Panth\'eon Sorbonne, Paris. France. Email: xavier.venel@univ-paris1.fr
}, Bruno Ziliotto\thanks{TSE (GREMAQ, Universit\' e Toulouse 1 Capitole), 21 all\' ee de Brienne, 31000
Toulouse, France.
}}


\maketitle

\begin{abstract}

%

In several standard models of dynamic programming (gambling houses, MDPs, POMDPs), we prove the existence of a robust notion of value for the infinitely repeated problem, namely the \textit{pathwise uniform value}. This solves two open problems. First, this shows that for any $\epsilon>0$, the decision-maker has a pure strategy $\sigma$ which is $\epsilon$-optimal in any $n$-stage game, provided that $n$ is big enough (this result was only known for behavior strategies, that is, strategies which use randomization). Second, the strategy $\sigma$ can be chosen such that under the long-run average payoff criterion,
the decision-maker has more than the limit of the $n$-stage values. 

\end{abstract}
%




\section*{Introduction}

The standard  model of Markov Decision Process (or Controlled Markov chain) was  introduced by Bellman \cite{Bellman_57} and has been extensively studied since then. In this model, at the beginning of every stage, a decision-maker perfectly observes the current state, and chooses an action accordingly, possibly randomly. The current state and the selected action determine a stage payoff and the law of the next state. There are two standard ways to aggregate the stream of payoffs. Given a strictly positive integer $n$, in the $n$-stage MDP, the total payoff is the Cesaro mean $n^{-1} \sum_{m=1}^n g_m$, where $g_m$ is the payoff at stage $m$. Given $\lambda \in (0,1]$, in the $\lambda$-discounted MDP, the total payoff is the $\lambda$-discounted sum $\lambda \sum_{m \geq 1} (1-\lambda)^{m-1} g_m$. The maximum payoff that the decision-maker can obtain in the $n$-stage problem (resp. $\lambda$-discounted problem) is denoted by $v_n$ (resp. $v_\lambda$). 

A huge part of the literature investigates \textit{long-term} MDPs, that is, MDPs which are repeated a large number of times. In the $n$-stage problem (resp. $\lambda$-discounted problem), this corresponds to $n$ being large (resp. $\lambda$ being small). A first approach is to determine whether $(v_n)$ and $(v_{\lambda})$ converge when $n$ goes to infinity and $\lambda$ goes to 0, and whether the two limits coincide. When this is the case, the MDP is said to have an \textit{asymptotic value}.  The asymptotic value represents the long-term payoff outcome. When the asymptotic value exists, a second approach consists in determining if for any $\epsilon>0$, there exists a behavior (resp. pure) strategy that is optimal up to $\epsilon$ in any $n$-stage and $\lambda$-discounted problem, provided that $n$ is big and $\lambda$ is small. When this is the case, the MDP is said to have a $\textit{uniform value}$ in behavior (resp. pure) strategies. 

A third approach is to define the payoff in the infinite problem as being the expectation of $\liminf_{n \rightarrow+\infty} n^{-1} \sum_{m=1}^n g_m$: in literature, this is referred as the \textit{long-run average payoff criterion}\footnote{In some papers, the decision-maker minimizes the cost: in this case, the long-run average payoff criterion corresponds to the \textit{long-run average cost criterion}.} (AP criterion,
see Arapostathis et al. \cite{SURVEYAC} for a review of the subject). We denote by $w_\infty$ the maximal payoff that the decision-maker can guarantee under this criterion. Clearly, under this criterion, the decision-maker cannot have more than $\liminf_{n \rightarrow+\infty} v_n$. A natural question is whether he can obtain $\liminf_{n \rightarrow+\infty} v_n$.

When the set space and action sets are finite, Blackwell \cite{Blackwell_62} has proved the existence of a pure strategy that is optimal for every discount factor close to 0, and one can deduce that the uniform value exists in pure strategies, and that under the AP criterion, the decision-maker can have $\lim_{n \rightarrow+\infty} v_n$.

In many situations, the decision-maker may not be perfectly informed of the current state variable. For instance, if the state variable represents a resource stock (like the amount of oil in an oil field), the quantity left, which represents the state, can be evaluated, but is not exactly known. This motivates the introduction of the more general model of Partially Observable Markov Decision Process (POMDP). In this model, at each stage, the decision-maker does not observe the current state, but instead receives a signal which is correlated to it. 
Rosenberg, Solan and Vieille \cite{RSV02} have proved that any POMDP has a uniform value in behavior strategies, when the state space, the action set and the signal set are finite. In the proof, the authors highlight the necessity that the decision-maker resort to behavior strategies, and ask whether the uniform value exists in pure strategies. They also raise the question of the behavior of the time averages of the payoffs, which is linked to the AP criterion. Renault \cite{R11} and Renault and Venel \cite{RV12} have provided two alternative proofs of the existence of the uniform value in behavior strategies in POMDPs, and also ask whether the uniform value exists in pure strategies. 

One of the main contributions of this paper is to solve this question positively. We prove that POMDPs have a uniform value in pure strategies. Moreover, for all $\epsilon>0$, under the AP criterion, the decision-maker can have $\lim_{n \rightarrow +\infty} v_n-\epsilon$. In fact, we prove this result in a much more general framework, as we shall see now.

The result of Rosenberg, Solan and Vieille \cite{RSV02} (existence of the uniform value in behavior strategies in POMDPs) has been generalized in several dynamic programming models with infinite state space and action set. The first one is to consider the model of gambling house. Introduced by Dubins and Savage \cite{Dubins_1965}, a gambling house is defined by a correspondence from a metric space $X$ to the set of probabilities on $X$. At every stage, the decision-maker chooses a probability on $X$ which is compatible with the correspondence and the current state. A new state is drawn from this probability, and this new state determines the stage payoff. When the state space is compact, and the correspondence is 1-Lipschitz, and the payoff function is continuous (for suitable metrics), the existence of the uniform value in behavior strategies stems from the main theorem in \cite{R11}. One can deduce from this result the existence of the uniform value in behavior strategies in MDPs and POMDPs, for a finite state space and any action and signal sets. Renault and Venel \cite{RV12} have extended the results of \cite{R11} to more general payoff evaluations. 

The proofs in Renault \cite{R11} and Renault and Venel \cite{RV12} are quite different from the one of Rosenberg, Solan and Vieille \cite{RSV02}. Still, they heavily rely on the use of behavior strategies for the decision-maker, and they do not provide any results concerning the AP criterion.

In this paper, we consider a gambling house with compact state space, closed graph correspondence and continuous payoff function. We show that if the family $\{v_n,n\geq 1\}$ is equicontinuous and $w_\infty$ is continuous, the gambling house has a uniform value in pure strategies. Moreover, for all $\epsilon>0$, the decision-maker can guarantee $\lim_{n \rightarrow+\infty} v_n-\epsilon$ under the AP criterion. This result especially applies to $1$-Lipschitz gambling houses. We deduce the same result for compact MDPs with 1-Lipschitz transition, and POMDPs with finite set space, compact action set and finite signal set.

Note that under an ergodic assumption on the transition function, like assuming that from any state, the decision-maker can make the state go back to the initial state (see Altman \cite{A94}), or assuming that the law of the state variable converges to an invariant measure (see Borkar \cite{Borkar_1988, Borkar_2000}), these results were already known. One remarkable feature of our proof is that we are able to use ergodic theory without any ergodic assumptions.

The paper is organized as follows. The first part presents the model of gambling house and recalls usual notions of value. The second part defines pathwise uniform value and states our results, that is, the existence of the pathwise uniform value in gambling houses, MDPs and POMDPs. The last three parts are dedicated to the proof of these results.


\section{Gambling houses}
\subsection{Model of gambling house}
Let us start with a few notations.
We denote by $\m{N}^*$ the set of strictly positive integers. If $A$ is a measurable space, we denote by $\Delta(A)$ the set of probability measures over $A$. If $(A,d)$ is a compact metric space, we will always equip $(A,d)$ with the Borelian algebra, and denote by $\mathcal{B}(A)$ the set of Borel subsets of $A$. The set of continuous functions from $A$ to $[0,1]$ is denoted by $\mathcal{C}(A,[0,1])$. The set $\Delta(A)$ is compact metric for the Kantorovich-Rubinstein distance $d_{KR}$, which metrizes the weak$^*$ topology. Recall that the distance $d_{KR}$ is defined for all $z$ and $z'$ in $\Delta(A)$ by
\[
d_{KR}(z,z'):=\sup_{f \in E_1} \left|\int_A f(x) z(dx)-\int_A f(x) z'(dx) \right|=\inf_{\pi \in \Pi(z,z')} \int_{A\times A} d(x,y) \pi(dx,dy),
\]
where $E_1\subset \mathcal{C}(A,[0,1])$ is the set of $1$-Lipschitz functions from $A$ to $[0,1]$ and $\Pi(z,z') \subset \Delta(A \times A)$ is the set of measures on $A \times A$ with first marginal $z$ and second marginal $z'$. Because $A$ is compact, the infimum is a minimum. For $f \in \mathcal{C}(A,[0,1])$, the linear extension of $f$ is the function $\hat{f} \in \mathcal{C}(\Delta(A),[0,1])$, defined for $z \in \Delta(A)$ by
\begin{equation*}
\hat{f}(z):=\int_{A} f(x) z(dx).
\end{equation*}

\noindent A \textit{gambling house} $\Gamma=(X,F,r)$ is defined by the following elements:
\begin{itemize}
\item $X$ is the \textit{state space}, which is assumed to be compact metric for some distance $d$.
\item $F:(X,d) \rightrightarrows (\Delta(X),d_{KR})$ is a correspondence with a closed graph and nonempty values.
\item $r:X \rightarrow [0,1]$ is the \textit{payoff function}, which is assumed to be continuous.
\end{itemize}

\begin{remark}
Because the state space is compact, $F$ is a closed graph correspondence if and only if it is an upper hemicontinuous correspondence with closed values.
\end{remark}




Let $x_0 \in X$ be an initial state. The gambling house starting from $x_0$ proceeds as follows. At each stage $m \geq 1$, the decision-maker chooses $z_{m} \in F(x_{m-1})$. A new state $x_{m}$ is drawn from the probability distribution $z_{m}$, and the decision-maker gets the payoff $r(x_{m})$. 


For the definition of strategies, we follow Maitra and Sudderth \cite[Chapter 2]{Maitra_96}. First, we need the following definition (see \cite[Chapter 11, section 1.8]{potentials}):
\begin{definition}
Let $\nu \in \Delta(\Delta(X))$. The \emph{barycenter} of $\nu$ is the probability measure $\mu=\bary(\nu) \in \Delta(X)$ such that for all $f \in \mathcal{C}(X,[0,1])$,
\[
\hat{f}(\mu)=\int_{\Delta(X)} \hat{f}(z) \nu(dz).
\]
Given $M$ a closed subset of $\Delta(X)$, we denote by $\sco M$ the \emph{strong convex hull} of the set  $M$, that is,
\begin{equation*}
\sco M :=\left\{\bary(\nu), \nu \in \Delta(M)\right\}.
\end{equation*}
Equivalently, $\sco M$ is the closure of the convex hull of $M$.
\end{definition}

For every $m \geq 1$, we denote by $H_m:=X^{m}$ the set of possible histories before stage $m$, which is compact for the product topology.

\begin{definition}
A \emph{behavior (resp. pure) strategy} $\sigma$ is a sequence of mappings $\sigma:=(\sigma_m)_{m \geq 1}$ such that for every $m\geq 1$,
\begin{itemize}
\item $\sigma_m:H_m \rightarrow \Delta(X)$ is (Borel) measurable, 
\item for all $h_m=(x_0,...,x_{m-1}) \in H_m$, $\sigma_m(h_m) \in \sco(F(x_{m-1}))$ (resp. 
$\sigma_m(h_m) \in F(x_{m-1})$).
\end{itemize}
We denote by $\Sigma$ (resp. $\Sigma_p$) the set of behavior (resp. pure) strategies.
\end{definition}
Note that $\Sigma_p \subset \Sigma$. The following proposition ensures that $\Sigma_p$ is nonempty. This is a special case of Kuratowski-Ryll-Nardzewski theorem (see \cite[Theorem 18.13, p. 600] {hitchhiker}.
\begin{proposition} \label{Dselection}
Let $K_1$ and $K_2$ be two compact metric spaces, and $\Phi: K_1 \rightrightarrows K_2$ be a closed graph correspondence with nonempty values. Then $\Phi$ admits a measurable selector, that is, there exists a measurable mapping $\phi: K_1 \rightarrow K_2$ such that for all $k \in K_1$, $\phi(k) \in K_2$.
\end{proposition}
\begin{proof}
In \cite{hitchhiker}, the theorem is stated for weakly measurable correspondences. By \cite[Theorem 18.10, p. 598]{hitchhiker} and \cite[Theorem 18.20, p. 606]{hitchhiker}, any correspondence satisfying the assumptions of the proposition is weakly measurable, thus the proposition holds. 
\end{proof}


\begin{definition}
A strategy $\sigma \in \Sigma$ is \emph{Markov} if there exists a measurable mapping $f: \m{N}^* \times X \rightarrow \Delta(X)$ such that for every $h_m=(x_0,...,x_{m-1}) \in H_m$, $\sigma(h_{m})=f(m,x_{m-1})$. When this is the case, we identify $\sigma$ with $f$.

A strategy $\sigma$ is \emph{stationary} if there exists a measurable mapping $f:X \rightarrow \Delta(X)$ such that for every $h_m=(x_0,...,x_{m-1}) \in H_m$, $\sigma(h_m)=f(x_{m-1})$. When this is the case, we identify $\sigma$ with $f$.
\end{definition}



Let $H_{\infty}:=X^{\m{N}}$ be the set of all possible plays in the gambling house $\Gamma$. 
By the Kolmogorov extension theorem, an initial state $x_0 \in X$ and a behavior strategy $\sigma$ determine a unique probability measure over $H_{\infty}$, denoted by $\m{P}^{x_0}_{\sigma}$.




Let $x_0 \in X$ and $n \geq 1$. The payoff in the $n$-stage problem starting from $x_0$ is defined for $\sigma \in \Sigma$ by
\[
\gamma_n(x_0,\sigma):=\m{E}^{x_0}_{\sigma}\left( \frac{1}{n} \sum_{m=1}^{n} r_m \right),
\]
where $r_m:=r(x_m)$ is the payoff at stage $m \in \m{N}^*$. 
The value $v_n(x_0)$ of this problem is the maximum expected payoff with respect to behavior strategies:
\[
v_n(x_0):=\sup_{\sigma \in \Sigma} \gamma_n(x_0,\sigma).
\]
By Feinberg \cite[Theorem 5.2]{F96}, any behavior strategy can be assimilated to a probability measure on the set of pure strategies. It follows that the above supremum is reached at a pure strategy. \\


\begin{remark} \label{DextensionF}
For $\mu \in \Delta(X)$, one can also define the gambling house with initial distribution $\mu$, where the initial state is drawn from $\mu$ and announced to the decision-maker. The definition of strategies and values are the same, and for all $n \in \m{N}^*$, the value of the $n$-stage gambling house starting from $\mu$ is equal to $\hat{v}_n(\mu)$. 
\end{remark}
\subsection{Long-term gambling houses}
\subsubsection{Uniform value}



\begin{definition}
Let $x_0 \in X$. The gambling house $\Gamma(x_0)$ has an \emph{asymptotic value} $v_{\infty}(x_0) \in [0,1]$ if the sequence $(v_n(x_0))_{n\geq 1}$ converges to $v_\infty(x_0)$.
\end{definition}

\begin{definition}
Let $x_0 \in X$. The gambling house $\Gamma(x_0)$ has a \emph{uniform value} $v_{\infty}(x_0) \in [0,1]$ in behavior (resp. pure) strategies if it has an asymptotic value $v_{\infty}(x_0)$ and for every $\varepsilon>0$, there exists $n_0 \in \m{N}^*$ and a behavior (resp. pure) strategy $\sigma$ such that for all $n\geq n_0$,
\[
\gamma_n(x_0,\sigma) \geq v_{\infty}(x_0)-\varepsilon.
\]
\end{definition}


\begin{definition}
A gambling house $\Gamma$ is $1$-Lipschitz if its correspondence $F$ is $1$-Lipschitz, that is, for every $x \in X$, every $u\in F(x)$ and every $y\in X$, there exists $w\in F(y)$ such that $d_{KR}(u,w) \leq d(x,y)$.
 \end{definition}
Renault and Venel \cite{RV12} have proved that any 1-Lipschitz gambling house has a  uniform value in behavior strategies\footnote{In fact, their model of gambling house is slightly different: they do not assume that $F$ is closed-valued, but instead assume that it takes values in the set of probability measures on $X$ with finite support.}. They asked about the existence of the uniform value in pure strategies. This is a recurring open problem in the literature. In the framework of POMDPs, this open problem already appeared in Rosenberg, Solan and Vieille \cite{RSV02} and in Renault \cite{R11}.

 
\subsubsection{The long-run average payoff criterion}
To study long-term dynamic programming problems, an alternative to the uniform approach is to associate a payoff to each infinite history. Given an initial state $x_0 \in X$, the \textit{infinitely repeated} gambling house $\Gamma_{\infty}(x_0)$ is the problem with strategy set $\Sigma$, and payoff function $\gamma_{\infty}$ defined for all $\sigma \in \Sigma$ by
\begin{equation*}
\gamma_{\infty}(x_0,\sigma):=\m{E}^{x_0}_{\sigma}\left(\liminf_{n \rightarrow +\infty} \frac{1}{n} \sum_{m=1}^n r_m \right).
\end{equation*}
In the literature, the above payoff is often referred as the \textit{long-run average payoff criterion} (see \cite{SURVEYAC}). The value of $\Gamma_{\infty}(x_0)$ is
\begin{equation*}
w_{\infty}(x_0):=\sup_{\sigma \in \Sigma} \gamma_{\infty}(x_0,\sigma).
\end{equation*}
\begin{remark}
The above supremum may not be reached: there may not exist 0-optimal strategies in $\Gamma_{\infty}(x_0)$ (see for example Rosenberg, Solan and Vieille \cite{RSV02}). 
\end{remark}
The following proposition plays a key role in this paper: 
\begin{proposition} \label{Dpurification}
For all $\epsilon>0$, there exists $\epsilon$-optimal pure strategies in $\Gamma_{\infty}(x_0)$.
\end{proposition}
\begin{proof}
Exactly like for the $n$-stage game, this result is a direct consequence of Theorem 5.2 in Feinberg \cite{F96}. 
\end{proof}
If $\Gamma(x_0)$ has a uniform value $v_{\infty}(x_0)$, we have $w_{\infty}(x_0) \leq v_{\infty}(x_0)$ by the dominated convergence theorem. A natural question is to ask whether the equality holds. When this is the case, it significantly strengthens the notion of uniform value, as shown by the following example.

\begin{example}
There are two states, $x$ and $x^*$, and $F(x)=F(x^*)=\left\{x,x^* \right\}$. Moreover, $r(x)=0$ and $r(x^*)=1$. Thus, at each stage, the decision-maker has to choose between having a payoff 0 and having a payoff 1. Obviously, this problem has a uniform value equal to 1. Let $\epsilon>0$. Let $\sigma$ be the strategy such that for all $n \in \m{N}$, at stage $2^{2^n}-1$, the decision-maker chooses $x$ with probability $\epsilon/2$, and sticks to this choice until stage $2^{2^{n+1}}-1$; with probability $1-\epsilon/2$, he chooses $x^*$, and sticks to this choice until stage $2^{2^{n+1}}-1$. The strategy $\sigma$ is uniformly $\epsilon$-optimal: there exists $n_0 \in \m{N}^*$ such that for all $n \geq n_0$, 
\begin{equation*}
\gamma_n(x,\sigma) \geq 1-\epsilon.
\end{equation*}
Nonetheless, by the law of large numbers, for any $n_0 \in \m{N}^*$, there exists a random time $T$ such that $\m{P}^{x}_{\sigma}$ almost surely, $T \geq n_0$ and
\begin{equation*}
\frac{1}{T} \sum_{m=1}^T r_m \leq \epsilon.
\end{equation*}
Therefore, the strategy $\sigma$ does not guarantee more than $\epsilon$ in the game $\Gamma_\infty(x)$.

\end{example}

\section{Main results}
\subsection{Gambling houses}


We introduce a stronger notion of uniform value, which allows us to deal with the two open questions mentioned in the previous section at the same time. 


\begin{definition} \label{Ddefpath}
Let $x_0 \in X$. The gambling house $\Gamma(x_0)$ has a \emph{pathwise uniform value} in behavior (resp. pure) strategies if 
\begin{itemize}
\item
The gambling house $\Gamma(x_0)$ has an asymptotic value $v_{\infty}(x_0)$.
\item
For all $\epsilon>0$, there exists a behavior (resp. pure) strategy $\sigma$ such that
\begin{equation*}
\gamma_{\infty}(x_0,\sigma) \geq v_{\infty}(x_0)-\epsilon.
\end{equation*}
\end{itemize}
A strategy $\sigma$ satisfying the above equation is called $\textit{pathwise $\epsilon$-optimal strategy}$. 
When for all $x_0 \in X$, $\Gamma(x_0)$ has a pathwise uniform value in behavior (resp. pure) strategies, we say that $\Gamma$ has a pathwise uniform value in behavior (resp. pure) strategies. 
\end{definition}
Proposition \ref{Dpurification} implies that there exists a pathwise uniform value in behavior strategies if and only if there exists a pathwise uniform value in pure strategies. The following proposition shows that the concept of pathwise uniform value is more general than the concept of uniform value. 
\begin{proposition}\label{pathuni}
Assume that $\Gamma(x_0)$ has a pathwise uniform value (in behavior or pure strategies). Then it has a uniform value in  pure strategies.
\end{proposition}

\begin{proof}
By Proposition \ref{Dpurification}, $\Gamma(x_0)$ has a pathwise uniform value in pure strategies. Let $\epsilon>0$, and $\sigma$ be a pathwise $\epsilon$-optimal pure strategy. We have
\begin{equation*}
\m{E}^{x_0}_{\sigma}\left(\liminf_{n \rightarrow +\infty} \frac{1}{n} \sum_{m=1}^n r_m \right) \geq v_{\infty}(x_0)-\epsilon.
\end{equation*}
By Fatou's lemma, it follows that
\begin{equation*}
\liminf_{n \rightarrow +\infty} \m{E}^{x_0}_{\sigma}\left( \frac{1}{n} \sum_{m=1}^n r_m \right) \geq v_{\infty}(x_0)-\epsilon,
\end{equation*}
and the gambling house $\Gamma(x_0)$ has a uniform value in pure strategies.

%


\end{proof}
We can now state our main theorem concerning gambling houses:
\begin{theorem}\label{Dtheo1}
 Let $\Gamma$ be a gambling house such that $\{v_n,n \geq 1\}$ is uniformly equicontinuous and $w_\infty$ is continuous. Then $\Gamma$ has a pathwise uniform value in pure strategies. Consequently, it has a uniform value in pure strategies, and
 \begin{equation*}
w_{\infty}=v_{\infty}.
\end{equation*}
\end{theorem}


In particular, we obtain the following result.

\begin{theorem}\label{D1Lips}
Let $\Gamma$ be a 1-Lipschitz gambling house. Then $\Gamma$ has a pathwise uniform value in pure strategies. Consequently, it has a uniform value in pure strategies, and
\begin{equation*}
w_{\infty}=v_{\infty}.
\end{equation*}

\end{theorem}

In the two next subsections, we present similar results for MDPs and POMDPs. 
\subsection{MDPs} \label{subMDP}
A Markov Decision Process (MDP) is a $4$-uple $\Gamma=(K,I,g,q)$, where $(K,d_K)$ is a compact metric state space, $(I,d_I)$ is a compact metric action set, $g:K \times I \rightarrow [0,1]$ is a continuous payoff function, and $q:K \times I \rightarrow \Delta(K)$ is a continuous transition function. As usual, the set $\Delta(K)$ is equipped wih the KR metric, and we assume that for all $i \in I$, $q(.,i)$ is $1$-Lipschitz.
Given an initial state $k_1 \in K$ known by the decision-maker, the MDP $\Gamma(k_1)$ proceeds as follows. At each stage $m \geq 1$, the decision-maker chooses $i_m \in I$, and gets the payoff $g_m:=g(k_m,i_m)$. A new state $k_{m+1}$ is drawn from $q(k_m,i_m)$, and is announced to the decision-maker. Then, $\Gamma(k_1)$ moves on to stage $m+1$. 
A behavior (resp. pure) strategy is a measurable map $\sigma : \cup_{m \geq 1}  K \times  (I \times K)^{m-1} \rightarrow \Delta(I)$ (resp. $\sigma : \cup_{m \geq 1} K \times (I \times K)^{m-1} \rightarrow I$). An initial state $k_1$ and a strategy $\sigma$ induce a probability measure $\m{P}^{k_1}_{\sigma}$ on the set of plays $H_{\infty}=(K \times I)^{\m{N}^*}$.

The notion of uniform value is defined in the same way as in gambling houses. 
We prove the following theorem:
\begin{theorem} \label{DthmMDP}
The MDP $\Gamma$ has a pathwise uniform value in pure strategies, that is, for all $k_1 \in K$, the two following statements hold: 
\begin{itemize}
\item
The sequence $(v_{n}(k_1))$ converges when $n$ goes to infinity to some real number $v_{\infty}(k_1)$.
\item
For all $\epsilon>0$, there exists a pure strategy $\sigma$ such that
\begin{equation*}
\m{E}^{k_1}_{\sigma}\left(\liminf_{n \rightarrow+\infty} \frac{1}{n} \sum_{m=1}^{n} g(k_m,i_m) \right) \geq v_{\infty}(k_1)-\epsilon.
\end{equation*}
\end{itemize}
Consequently, the MDP $\Gamma$ has a uniform value in pure strategies.
\end{theorem}

\subsection{POMDPs} \label{subPOMDP}
A Partially Observable Markov Decision Process (POMDP) is a $5$-uple $\Gamma=(K,I,S,g,q)$, where $K$ is a finite set space, $I$ is a compact metric action set, $S$ is a finite signal set, $g:K \times I \rightarrow [0,1]$ is a continuous payoff function, and $q:K \times I \rightarrow \Delta(K \times S)$ is a continuous transition function. Given an initial distribution $p_1 \in \Delta(K)$, the POMDP $\Gamma(p_1)$ proceeds as follows. An initial state $k_1$ is drawn from $p_1$, and the decision-maker is not informed about it. At each stage $m \geq 1$, the decision-maker chooses $i_m \in I$, and gets the (unobserved) payoff $g(k_m,i_m)$. A pair $(k_{m+1},s_m)$ is drawn from $q(k_m,i_m)$, and the decision-maker receives the signal $s_m$. Then the game proceeds to stage $m+1$. 
A behavior strategy (resp. pure strategy) is a measurable map $\sigma : \cup_{m \geq 1} (I \times S)^{m-1} \rightarrow \Delta(I)$ (resp. $\sigma : \cup_{m \geq 1} (I \times S)^{m-1} \rightarrow I$). An initial distribution $p_1 \in \Delta(K)$ and a strategy $\sigma$ induce a probability measure $\m{P}^{p_1}_{\sigma}$ on the set of plays $H_{\infty}:=(K \times I \times S)^{\m{N}^*}$.

The notion of uniform value is defined in the same way as in gambling houses. We prove the following theorem:

\begin{theorem} \label{DthmPOMDP}
The POMDP $\Gamma$ has a pathwise uniform value in pure strategies, that is, for all $p_1 \in \Delta(K)$, the two following statements hold:
\begin{itemize}
\item
The sequence $(v_{n}(p_1))$ converges when $n$ goes to infinity to some real number $v_{\infty}(p_1)$.
\item
For all $\epsilon>0$, there exists a pure strategy $\sigma$ such that
\begin{equation*}
\m{E}^{p_1}_{\sigma}\left(\liminf_{n \rightarrow+\infty} \frac{1}{n} \sum_{m=1}^{n} g(k_m,i_m) \right) \geq v_{\infty}(p_1)-\epsilon.
\end{equation*}
\end{itemize}
Consequently, the POMDP $\Gamma$ has a uniform value in pure strategies.

\end{theorem}


In particular, this theorem solves positively the open question mentioned in \cite{RSV02}, \cite{R11} and \cite{RV12}: finite POMDPs have a uniform value in pure strategies.  
\\




\section{Proof of Theorem \ref{Dtheo1}} \label{Dprooftheo1}

Let $\Gamma=(X,F,r)$ be a gambling house such that $\{v_n,n \geq 1\} \cup \{w_\infty\}$ is uniformly equicontinuous. Let $v:X \rightarrow [0,1]$ be defined by $v:=\limsup_{n \rightarrow+\infty} v_n$.

Let $x_0 \in X$ be an initial state. By Proposition \ref{Dpurification}, in order to prove Theorem $\ref{Dtheo1}$, it is sufficient to prove that for all $\epsilon>0$, there exists a behavior strategy $\sigma$ such that 
\begin{equation*}
\gamma_{\infty}(x_0,\sigma)=\m{E}^{x_0}_{\sigma}\left(\liminf_{n \rightarrow+\infty} \frac{1}{n} \sum_{m=1}^{n} r_m \right) \geq v(x_0)-\epsilon.
\end{equation*}
Let us first give the structure and the intuition of the proof. It builds on three main ideas, each of them corresponding to a lemma. 

First, Lemma \ref{Dinvariant} associates to $x_0$ a probability measure $\mu^* \in \Delta(X)$, such that:
\begin{itemize}
\item
Going from $x_0$, for all $\epsilon>0$ and $n_0 \in \m{N}^*$, there exists a strategy $\sigma_0$ and $n \geq n_0$ such that the occupation measure $\frac{1}{n} \sum_{m=1}^n z_m \in \Delta(X)$ is close to $\mu^*$ up to $\epsilon$ (for the KR distance).
\item
$\hat{r}(\mu^*)=\hat{v}(\mu^*)=v(x_0)$
\item
If the initial state is drawn according to $\mu^*$, the decision-maker has a behavior stationary strategy $\sigma^*$ such that for all $m \geq 1$, $z_m$ is distributed according to $\mu^*$ ($\mu^*$ is an invariant measure for the gambling house).
\end{itemize}
Let $x$ be in the support of $\mu^*$. Building on a pathwise ergodic theorem, Lemma \ref{Dbirkhoff} shows that 
\begin{equation*}
\frac{1}{n} \sum_{m=1}^{n} r_m \rightarrow v(x) \quad \m{P}_{ \sigma^*}^x \ \text{a.s.}
\end{equation*}
Let $y \in X$ be close to $x$. Lemma \ref{Djunction_opt} shows that, if $y \in X$ is close to $x$, then there exists a behavior strategy $\sigma$ such that $\gamma_{\infty}(y,\sigma)$ is close to $v(y)$. 
\\

These lemmas are put together in the following way. Lemma \ref{Dinvariant} implies that, going from $x_0$, the decision-maker has a strategy $\sigma_0$ such that there exists a (deterministic) stage $m \geq 1$ such that with high probability, the state $x_m$ is close to the support of $\mu^*$, and such that the expectation of $v(x_m)$ is close to $v(x_0)$. Let $x$ be an element in the support of $\mu^*$ such that $x_m$ is close to $x$. By Lemma \ref{Djunction_opt}, going from $x_m$, the decision-maker has a strategy $\sigma$ such that $\gamma_{\infty}(x_m,\sigma)$ is close to $v(x_m)$. Let $\widetilde{\sigma}$ be the strategy that plays $\sigma_0$ until stage $m$, then switches to $\sigma$. Then $\gamma_{\infty}(x_0,\widetilde{\sigma})$ is close to $v(x_0)$, which concludes the proof of Theorem \ref{Dtheo1}.

\subsection{Preliminary results} \label{Dpreliminary}

Let $\Gamma=(X,F,r)$ be a gambling house. We define a relaxed version of the gambling house, in order to obtain a deterministic convex gambling house $H: \Delta(X) \rightrightarrows \Delta(X)$. The interpretation of $H(z)$ is the following: if the initial state is drawn according to $z$, $H(z)$ is the set of all possible measures on the next state that the decision-maker can generate by using behavior strategies.

 First, we define $G: X \rightrightarrows \Delta(X)$ by
\begin{align*}
\forall x \in X \quad G(x) & := {\rm \sco}(F(x)).
\end{align*}

By \cite[Theorem 17.35, p.573]{hitchhiker}, the correspondence $G$ has a closed graph, which is denoted by $\graph G$. Note that a behavior strategy in the gambling house $\Gamma$ corresponds to a pure strategy in the gambling house $(X,G,r)$.
For every $z\in \Delta(X)$, we define $H(z)$ by
\begin{align*}
H(z) & :=\left\{ \vphantom{\int_{\Delta(X)}} \mu \in \Delta(X) \ | \ \exists \ \sigma:X \rightarrow \Delta(X) \ \text{measurable} \ \text{s.t.} \ \forall x\in X, \ \sigma(x)\in G(x) \ \text{and} \	\right. \\
 & \hspace{5mm} \left. \forall f \in \mathcal{C}(X,[0,1]), \ \hat{f}(\mu)=\int_{X} \hat{f}(\sigma(x)) z(dx) \right\}.
\end{align*}
Note that replacing ``$\forall x\in X, \ \sigma(x)\in G(x)$'' by ``$\forall x\in X, \ \sigma(x)\in G(x) \ z-a.s. $'' does not change the above definition (throughout the paper, ``a.s.'' stands for ``almost surely'').

By Proposition \ref{Dselection}, $H$ has nonempty values. We now check that the correspondence $H$ has a closed graph.

\begin{proposition} \label{Dhclosed}
The correspondence $H$ has a closed graph.
\end{proposition}

\begin{proof}
Let $(z_n,\mu_n)_{n \in \m{N}} \in (\graph H)^{\m{N}}$ such that $(z_n,\mu_n)_{n \in \m{N}}$ converges to some $(z,\mu) \in \Delta(X) \times \Delta(X)$. 
Let us show that $\mu \in H(z)$. For this, we construct $\sigma:X \rightarrow \Delta(X)$ associated to $\mu$ in the definition of $H(z)$. \\

By definition of $H$, for every $n\in \m{N}$, there exists $\sigma_n:X\rightarrow \Delta(X)$ a measurable selector of $G$ such that 
for every $f\in \mathcal{C}(X,[0,1])$,
\[
\hat{f}(\mu_n)=\int_{X} \hat{f}(\sigma_n(x))z_n(dx).
\]
Let $\pi_n \in \Delta(\graph G)$ such that the first marginal of $\pi_n$ is $z_n$, and the conditional distribution of $\pi_n$ knowing $x \in X$ is $\delta_{\sigma_n(x)} \in \Delta(\Delta(X))$.
By definition, for every $f\in \mathcal{C}(X,[0,1])$, we have
\begin{eqnarray*}
\int_{X\times \Delta(X)} \hat{f}(p) \pi_n(dx,dp)&=& \int_{X} \left( \int_{\Delta(X)} \hat{f}(p) \delta_{\sigma_n(x)}(dp) \right) z_n(dx)
\\
&=& \int_X \hat{f}(\sigma_n(x)) z_n(dx)
\\
&=& \hat{f}(\mu_n).
\end{eqnarray*}

The set $\Delta(\graph G)$ is compact, thus there exists $\pi$ a limit point of the sequence $(\pi_n)_{n \in \m{N}}$. By definition of the weak* topology on $\Delta(X)$ and on $\Delta(\graph G)$, the previous equation yields
\begin{equation} \label{Detage}
\int_{X\times \Delta(X)} \hat{f}(p) \pi(dx,dp)=\hat{f}(\mu).
\end{equation}

To conclude, let us disintegrate $\pi$. Let $z'$ be the first marginal of $\pi$. The sets $X$ and $\Delta(X)$ are compact metric spaces, thus there exists a probability kernel $K: X \times \mathcal{B}(\Delta(X)) \rightarrow [0,1]$ such that
\begin{itemize}
\item for every $x\in X$, $K(x,.) \in \Delta(\Delta(X))$,
\item for every $B \in \mathcal{B}(\Delta(X))$, $K(.,B)$ is measurable,
\item for every $h\in \mathcal{C}(X\times \Delta(X),[0,1])$,
\begin{equation} \label{Ddefdis}
\int_{X\times \Delta(X)} h(x,p) \pi(dx,dp)= \int_{X} \left( \int_{\Delta(X)} h(x,p) K(x,dp) \right) z'(dx).
\end{equation}
\end{itemize}
Note that the second condition is equivalent to: ``The mapping $x \rightarrow K(x,.)$ is measurable'' (see \cite[Proposition 7.26, p.134]{bertsekas}). 
For every $n \geq 1$, the first marginal of $\pi_n$ is equal to $z_n$ that converges to $z$, thus $z'=z$.
Define a measurable mapping $\sigma: X \rightarrow \Delta(X)$ by $\sigma(x):=\bary(K(x,.)) \in \Delta(X)$. Because $\pi \in \Delta(\graph G)$, we have $\sigma(x) \in G(x) \ z-\text{a.s}$. Let $f \in \mathcal{C}(X,[0,1])$. Using successively (\ref{Detage}) and (\ref{Ddefdis}) yield
\begin{eqnarray*}
\hat{f}(\mu)&=&\int_{X\times \Delta(X)} \hat{f}(p) \pi(dx,dp)
\\
&=& \int_X \left(\int_{\Delta(X)} \hat{f}(p) K(x,dp) \right) z(dx)
\\
&=&\int_X \hat{f}(\sigma(x)) z(dx).
\end{eqnarray*}
Thus, $\mu \in H(z)$, and $H$ has a closed graph.
\end{proof}

Let $\mu,\mu' \in \Delta(X)$. Denote $\lambda \cdot \mu+(1-\lambda) \cdot \mu'$ the probability measure $\mu'' \in \Delta(X)$ such that for all $f\in \mathcal{C}(X,[0,1])$,
\begin{equation*}
\hat{f}(\mu'')=\lambda \hat{f}(\mu)+(1-\lambda) \hat{f}(\mu').
\end{equation*}
For $(\mu_m)_{m \in \m{N}^*} \in \Delta(X)^{\m{N}^*}$ and $n \in \m{N}^*$, the measure $\displaystyle \frac{1}{n} \sum_{m=1}^n \mu_m$ is defined in a similar way. 

\begin{proposition} \label{Dlinear}
The correspondence $H$ is linear on $\Delta(X)$:
\[
\forall z,z' \in \Delta(X), \ \forall \lambda \in[0,1], \ H(\lambda \cdot z +(1-\lambda) \cdot z')=\lambda \cdot H(z)+ (1-\lambda) \cdot H(z').
\]
\end{proposition}

\begin{proof}
Let $z,z'\in \Delta(X)$ and $\lambda \in[0,1]$, then the inclusion
\[
\ H(\lambda \cdot z +(1-\lambda) \cdot z') \subset \lambda \cdot H(z)+ (1-\lambda) \cdot H(z')
\]
is immediate. We now prove the converse inclusion. Let $\mu \in \lambda \cdot H(z)+ (1-\lambda) \cdot H(z')$. By definition, there exists $\sigma:X\rightarrow \Delta(X)$ and $\sigma':X\rightarrow \Delta(X)$ two measurable selectors of $G$ such that for every $f\in \mathcal{C}(X,[0,1])$,
\begin{equation*}
\hat{f}(\mu)=\lambda \int_{X} \hat{f}(\sigma(x)) z(dx) +(1-\lambda) \int_{X} \hat{f}(\sigma'(x)) z'(dx).
\end{equation*}
Denote by $\pi$ (resp. $\pi'$), the probability distribution on $X\times \Delta(X) $ generated by $z$ and $\sigma$ (resp. $z'$ and $\sigma'$). Let $\pi'':=\lambda \cdot \pi + (1-\lambda) \cdot \pi'$, then $\pi''$ is a probability on $X\times \Delta(X)$ such that $\pi''(\graph(G))=1$, and the marginal on $X$ is $\lambda \cdot z +(1-\lambda) \cdot z'$. Let $\sigma'':X \rightarrow \Delta(X)$ given by the disintegration of $\pi''$ with respect to the first coordinate. Let $f\in \mathcal{C}(X,[0,1])$.
As in the proof of Proposition \ref{Dhclosed} (see Equation (\ref{Detage})), we have
\begin{eqnarray*}
\hat{f}(\mu)&=&\lambda \int_{X \times \Delta(X)} \hat{f}(p) \pi(dx,dp) +(1-\lambda) \int_{X \times \Delta(X)} \hat{f}(p) \pi'(dx,dp)
\\
&=& \int_{X \times \Delta(X)} \hat{f}(p) \pi''(dx,dp)
\\
&=&
\int_{X} \hat{f}(\sigma''(x)) z(dx),
\end{eqnarray*}
thus $\mu \in H(\lambda \cdot z+(1-\lambda) \cdot z')$.

\end{proof}

\subsection{Invariant measure}
The first lemma associates a fixed point of the correspondence $H$ to each initial state:


\begin{lemma} \label{Dinvariant}
Let $x_0 \in X$.
There exists a distribution $\mu^* \in \Delta(X)$ such that
\begin{itemize}
\item $\mu^*$ is $H$-invariant: $\mu^* \in H(\mu^*)$,
\item for every $\varepsilon>0$ and $N\geq 1$, there exists a (pure) strategy $\sigma_0$ and $n \geq N$ such that $\sigma$ is $0$-optimal in $\Gamma_n(x_0)$, $v_{n}(x_0) \geq v(x_0)-\epsilon$ and 
\[
d_{KR}\left(\frac{1}{n} \sum_{m=1}^n z_m(x_0,\sigma), \mu^* \right) \leq \varepsilon,
\]
where $z_m(x_0,\sigma_0) \in \Delta(X)$ is the distribution of $x_m$, the state at stage $m$, given the initial state $x_0$ and the strategy $\sigma_0$.
\item $\hat{r}(\mu^*)=\hat{v}(\mu^*)=v(x_0)$.
\end{itemize}
\end{lemma}

\begin{proof}
The proof builds on the same ideas as in Renault and Venel \cite[Proposition 3.24, p. 28]{RV12}.
Let $n \in \m{N}^*$ and $\sigma_0$ be a pure optimal strategy in the $n$-stage problem $\Gamma_n(x_0)$. 

Let 
\[
z_n:=\frac{1}{n}\sum_{m=1}^n z_m(x_0,\sigma_0),
\]
and
\[
z'_n:=\frac{1}{n}\sum_{m=2}^{n+1} z_m(x_0,\sigma_0).
\]

By construction, for every $m\in \{1,2,...,n\}$, $z_{m+1}(x_0,\sigma_0) \in H(z_m(x_0,\sigma_0))$, therefore by linearity of $H$ (see Proposition \ref{Dlinear})
\[
z_n' \in H(z_n).
\]
Moreover, we have
\begin{equation} \label{Dsamelimit}
d_{KR}(z_n,z'_n) \leq \frac{2}{n} \diam(X),
\end{equation}
where $\diam(X)$ is the diameter of $X$.

The set  $\Delta(X)$ is compact. Up to taking a subsequence, there exists $\mu^* \in \Delta(X)$ such that $(v_n(x_0))$ converges to $v(x_0)$ and $(z_n)$ converges to $\mu^*$. By inequality (\ref{Dsamelimit}), $(z_n')$ also converges to $\mu^*$. Because $H$ has a closed graph, we have $\mu^* \in H(\mu^*)$, and $\mu^*$ is $H$-invariant.  By construction, the second property is immediate.\\

\noindent Finally, we have a series of inequalities that imply the third property.

\begin{itemize}
\item $v$ is decreasing in expectation along trajectories: the sequence 
\\
$(\hat{v}(z_m(x_0,\sigma_0)))_{m\geq 1}$ is decreasing, thus for every $n\geq 1$,
\[
v(x_0) \geq \frac{1}{n} \sum_{m=1}^n \hat{v}(z_m(x_0,\sigma_0)) = \hat{v}(z_n).
\]
Taking $n$ to infinity, by continuity of $\hat{v}$, we obtain that $v(x_0) \geq \hat{v}(\mu^*).$

\item We showed that $\mu^* \in H(\mu^*)$. Let $\sigma^*:X \rightarrow \Delta(X)$ be the corresponding measurable selector of $G$. Let us consider the gambling house $\Gamma(\mu^*)$, where the initial state is drawn from $\mu^*$ and announced to the decision-maker (see Remark \ref{DextensionF}). The map $\sigma^*$ is a stationary strategy in $\Gamma(\mu^*)$, and for all $m \geq 1$, $z_m(\mu^*,\sigma^*)=\mu^*$. Consequently, for  
all $n \in \m{N}^*$, the strategy $\sigma^*$ guarantees $\hat{r}(\mu^*)$ in $\Gamma_n(\mu^*)$. Thus, we have
\[
\hat{v}(\mu^*) \geq \hat{r}(\mu^*).
\]
\item By construction, the payoff is linear on $\Delta(X)$ and $\hat{r}(z_n)=v_n(x_0)$. By continuity of $\hat{r}$, taking $n$ to infinity, we obtain
\[
\hat{r}(\mu^*) = v(x_0).
\]
\end{itemize}
\end{proof}
In the next section, we prove that in $\Gamma(\mu^*)$, under the strategy $\sigma^*$, the average payoffs converge almost surely to $v(x)$, where $x$ is the initial (random) state. 

\subsection{Pathwise ergodic theorem}
We recall here the ergodic theorem in Hern\'{a}ndez-Lerma and Lasserre \cite[Theorem 2.5.1, p. 37]{lasserre03}. 
\begin{theorem}[pathwise ergodic theorem] \label{Dergodic}
Let $(X,\mathcal{B})$ be a measurable space, and $\xi$ be a Markov chain on $(X,\mathcal{B})$, with transition probability function $P$. Let $\mu$ be an invariant probability measure for $P$. For every $f$ an integrable function with respect to $\mu$, there exist a set $B_f \in \mathcal{B}$ and a function $f^*$ integrable with respect to $\mu$, such that $\mu(B_f)=1$, and for all $x \in B_f$,
\begin{equation*}
\frac{1}{n} \sum_{m=1}^{n} f(\xi_m) \rightarrow f^*(\xi_0) \quad P_x-a.s.
\end{equation*}
Moreover,
\begin{equation*}
\int_X f^*(x) \mu(dx)=\int_X f(x) \mu(dx).
\end{equation*}
\end{theorem}

%

\begin{lemma}\label{Dbirkhoff}
Let $x_0 \in X$ and $\mu^* \in \Delta(X)$ be the corresponding invariant measure (see Lemma \ref{Dinvariant}).
There exist a measurable set $B \subset \Delta(X)$ such that $\mu^*(B)=1$ and a stationary strategy $\sigma^*: X \rightarrow \Delta(X)$ such that for all $x \in B$,
\begin{equation*}
\frac{1}{n} \sum_{m=1}^{n} r_m \rightarrow v(x) \quad \m{P}^x_{\sigma^*}-a.s.
\end{equation*}
\end{lemma}
\begin{proof}
Because $\mu^*$ is a fixed point of $H$, there exists  $\sigma^*:X\rightarrow \Delta(X)$ a measurable selector of $G$ (thus, a behavior stationary strategy in $\Gamma$) such that for all $f \in \mathcal{C}(X,[0,1])$,
\[
\hat{f}(\mu^*)=\int_{X} \hat{f}(\sigma^*(x)) \mu^*(dx).
\]
Consider the gambling house $\Gamma(\mu^*)$. 
Under $\sigma^*$, the sequence of states
$(x_m)_{m \in \m{N}}$ is a Markov chain with invariant measure $\mu^*$. From Theorem \ref{Dergodic}, there exist a measurable set $B_0 \subset X$ such that $\mu^*(B_0)=1$, and a measurable map $w:X \rightarrow [0,1]$ such that for all $x \in B_0$, we have
\begin{equation*}
\frac{1}{n} \sum_{m=1}^{n} r(x_m) \underset{n \rightarrow +\infty}{\rightarrow} w(x) \quad \m{P}^x_{\sigma^*}-\text{almost  surely,}
\end{equation*}
and
\begin{equation*}
\hat{w}(\mu^*)=\hat{r}(\mu^*).
\end{equation*}

\noindent We now prove that $w=v \quad \m{P}^{\mu^*}_{\sigma^*}-\text{a.s.}$. First, we prove that 
$w \leq v \quad \m{P}^{\mu^*}_{\sigma^*}-\text{a.s.}$. Let $x \in B_0$. Using first the dominated convergence theorem, then the definition of $v_n(x)$, we have
\begin{eqnarray*}
w(x)&=&\m{E}^x_{\sigma^*} \left( \lim_{n\rightarrow +\infty} \frac{1}{n} \sum_{m=1}^n r(x_m)\right)
\\
&=& \lim_{n\rightarrow +\infty} \m{E}^x_{\sigma^*} \left(  \frac{1}{n} \sum_{m=1}^n r(x_m)\right)
\\
&\leq& \limsup_{n\rightarrow +\infty} v_n(x)= v(x).
\end{eqnarray*}
Moreover, we know by Lemma \ref{Dinvariant} that $\hat{r}(\mu^*)=\hat{v}(\mu^*)$, therefore 
\[
\hat{w}(\mu^*) =\hat{r}(\mu^*) =\hat{v}(\mu^*).
\]
This implies that $w=v \quad \m{P}^{\mu^*}_{\sigma^*}-\text{a.s.}$, and the lemma is  proved. 
\end{proof}

\subsection{Junction lemma} \label{Djunction_section}

{\color{black}
By assumption, $\{v_n, n\geq 1\}\cup \{w_\infty\}$ is uniformly equicontinuous. Therefore, there exists an increasing modulus of continuity $\eta: \m{R}_{+} \rightarrow \m{R}_{+}$ such that
\[
\forall x,y\in X, \ |w_\infty(x)-w_\infty(y)| \leq \eta(d(x,y)),
\]
and for all $n\geq 1$,
\[
\forall x,y\in X, \ |v_n(x)-v_n(y)| \leq \eta(d(x,y)).
\]
Then, $v$ is also uniformly continuous with the same modulus of continuity.

\begin{lemma}\label{Djunction_opt}
Let $\varepsilon>0$, $x,y \in X$ and $\sigma^*$ be a strategy such that
\begin{equation*}
\frac{1}{n} \sum_{m=1}^{n} r_m \rightarrow v(x) \quad \m{P}^x_{\sigma^*} \ a.s.
\end{equation*}
Then there exists a strategy $\sigma$ such that 
\begin{equation*}
\m{E}^{y}_\sigma\left(\liminf_{n \rightarrow+\infty} \frac{1}{n} \sum_{m=1}^{n} r_m \right) \geq  v(y)-2\eta(d(x,y))-\varepsilon.
\end{equation*}
\end{lemma}

\begin{proof}
By assumption, we have
\[
\m{E}^{x}_{\sigma^*}\left(\liminf_{n \rightarrow+\infty} \frac{1}{n} \sum_{m=1}^{n} r_m \right)=\m{E}^{x}_{\sigma^*}\left( v(x) \right)=v(x),
\]
therefore $v(x) \geq w_\infty(x)$. Moreover, by Fatou's lemma, $w_\infty(x)\leq v(x)$. Thus, $w_\infty(x) = v(x)$. \\

Let $\varepsilon>0$. By definition of $w_\infty(y)$, there exists a strategy $\sigma$ such that

\begin{align*}
\m{E}^{y}_\sigma\left(\liminf_{n \rightarrow+\infty} \frac{1}{n} \sum_{m=1}^{n} r_m \right) & \geq  w_\infty(y)-\varepsilon, \\
& \geq w_{\infty}(x)-\eta(d(x,y))-\varepsilon, \\
& = v(x)-\eta(d(x,y))-\varepsilon, \\
& \geq v(y)-2\eta(d(x,y))-\varepsilon.
\end{align*}

\end{proof}

}

We can now finish the proof of Theorem \ref{Dtheo1}. 
\subsection{Conclusion of the proof}\label{concl}
\begin{proof}[Proof of Theorem \ref{Dtheo1}]
We can now put Lemma \ref{Dinvariant}, \ref{Dbirkhoff} and \ref{Djunction_opt} together to finish the proof of Theorem \ref{Dtheo1}. Fix an initial state $x_0 \in X$ and $\epsilon>0$. 
We will define a strategy $\widetilde{\sigma}$ as follows: start by following a strategy $\sigma_0$ until some stage $n_3$, then switch to another strategy depending on the state $x_{n_3}$. We first define the stage $n_3$, then build the strategy $\widetilde{\sigma}$ and finally check that this strategy indeed guarantees a good long-run average payoff. \\

{\color{black} By assumption, the family $(v_n)_{n \geq 1}$ is uniformly equicontinuous. Consequently, there exists $n_0 \in \m{N}^*$ such that for all $n \geq n_0$ and for all $x \in X,$
\begin{equation*}
v_n(x) \leq v(x)+\epsilon.
\end{equation*}}

We first consider Lemma \ref{Dinvariant} for $x_0$, $\epsilon'=\epsilon^3$ and $N=2 n_0$. There exists $\mu^*$ an invariant measure, $\sigma_0$ a (pure) strategy and $n_1\geq 2 n_0$ such that
$\mu^*$ satisfies the conclusion of Lemma \ref{Dinvariant} and
\begin{equation*}
d_{KR} \left(\frac{1}{n_1} \sum_{m=1}^{n_1} z_m(x_0,\sigma_0) ,\mu^* \right) \leq \epsilon^3.
\end{equation*}

Let $B$ be given by Lemma \ref{Dbirkhoff}. In general, there is no hope to prove the existence of a stage $m$ such that $z_m(x_0,\sigma_0)$ is close to $\mu^*$. Instead, we prove the existence of a stage $n_3$ such that under the strategy $\sigma_0$, $x_{n_3}$ is with high probability close to $B$, and $v(z_{n_3}(x_0,\sigma_0))$ is close to $v(x_0)$.


Let $n_2=\lfloor \epsilon n_1 \rfloor+1$, $A=\left\{x \in X|d(x,B) \leq \varepsilon \right\}$ and $A^c=\left\{x \in X|d(x,B) > \varepsilon \right\}$. We denote $\mu_{n_1}= \frac{1}{n_1} \sum_{m=1}^{n_1} z_m(x_0,\sigma_0)$. By property of the KR distance, there exists a coupling $\gamma \in \Delta(X \times X)$ such that the first marginal of $\gamma$ is $\mu_{n_1}$, the second marginal is $\mu^*$, and
\begin{eqnarray*}
d_{KR}(\mu_{n_1},\mu^*) & = & \int_{X^2} d(x,x') \gamma(dx,dx'). 
\end{eqnarray*}
By definition of $A$, for all $(x,x') \in A^c \times B$, we have $d(x,x')> \epsilon$. Thus, Markov inequality yields
\begin{eqnarray*}
\int_{X^2} d(x,x') \gamma(dx,dx') &\geq& \epsilon \gamma(A^c \times B)
\\
& =&  \varepsilon \mu_{n_1}(A^c).
\end{eqnarray*}
We deduce that $\mu_{n_1}(A^c)\leq \varepsilon^2$.
\noindent Because the $n_2$ first stages have a weight of order $\varepsilon$ in $\mu_{n_1}$, we deduce the existence of a stage $m$ such that $z_{m}(A^c) \leq \epsilon$:
\begin{eqnarray*}
\mu_{n_1}(A^c)&=&\frac{1}{n_1} \sum_{m=1}^{n_1} z_m(A^c)
\\
&=&\frac{1}{n_1} \sum_{m=1}^{n_2} z_m(A^c)+\frac{1}{n_1} \sum_{m=n_2+1}^{n_1} z_m(A^c)
\\
&\geq& \epsilon \min_{1 \leq m \leq n_2} z_m(A^c),
\end{eqnarray*}
and thus
\begin{align} \label{Dtoto_1}
z_{n_3}(A^c):=\min_{1 \leq m \leq n_2} z_m(A^c) \leq \epsilon.
\end{align}

\noindent Moreover, $\hat{v}(z_{n_3}(x_0,\sigma_0))$ is greater than $v(x_0)$ up to a margin $\varepsilon$. Indeed we have
\begin{align*}
\hat{v}(z_{n_3}(x_0,\sigma_0)) & \geq v_{n_1-n_3+1}(z_{n_3}(x_0,\sigma_0))- \varepsilon \\
       & \geq v_{n_1}(x_0)-\frac{n_3-1}{n_1}-\varepsilon \\
        & \geq v(x_0)-2\varepsilon-\varepsilon.\\
        & \geq v(x_0)-3\varepsilon.
\end{align*}

\noindent Using Equation (\ref{Dtoto_1}) and the last inequality, we deduce that 
\begin{equation*}
\m{E}^{x_0}_{\sigma_0}(1_{A} v(x_{n_3})) \geq \m{E}^{x_0}_{\sigma_0}(v(x_{n_3}))-z_{n_3} (A^c) \geq v(x_0)-4 \varepsilon.
\end{equation*}

We have defined both the initial strategy $\sigma_0$ and the switching stage $n_3$. To conclude, we use Lemma \ref{Djunction_opt} in order to define the strategy from stage $n_3$. Note that in Lemma \ref{Djunction_opt}, we did not prove  that the strategy $\sigma$ could be selected in a measurable way with respect to the state. Thus, we need to use a finite approximation.  The set $X$ is a compact metric set, thus there exists a partition $\{\mathcal{P}^1,...,\mathcal{P}^L\}$ of $X$ such that for every $l\in \{1,...,L\}$, $\mathcal{P}^l$ is measurable and $\diam(P^l) \leq \epsilon$. It follows that there exists a finite subset $\{x^1,...,x^L\}$ of $B$ such that for every $x\in A \cap \mathcal{P}^l$, $d(x,x^l)\leq 3 \varepsilon$. We denote by $\psi$ the application which associates to every $x\in A \cap \mathcal{P}^l$ the state $x^l$. 

We define the strategy $\widetilde{\sigma}$ as follows:
\begin{itemize}
\item Play $\sigma_0$ until stage $n_3$. 
\item If $x_{n_3} \in A$, then there exists $l \in \left\{1,...,L \right\}$ such that $x_{n_3} \in \mathcal{P}^l$. Play the strategy given by Lemma \ref{Djunction_opt}, with $x=x^l$ and $y=x_{n_3}$. If $x_{n_3} \notin A$, play any strategy.
\end{itemize}

Let us check that the strategy $\widetilde{\sigma}$ guarantees a good payoff with respect to the long-run average payoff criterion. By definition, we have
\begin{align*}
\gamma_{\infty}(x_0,\widetilde{\sigma}) &= \m{E}^{x_0}_{\widetilde{\sigma}} \left(\liminf_{n \rightarrow+\infty} \frac{1}{n} \sum_{m=1}^{n} r_m \right) \\
&= \m{E}^{x_0}_{\widetilde{\sigma}} \left( \m{E}^{x_{0}}_{\widetilde{\sigma}} \left(\liminf_{n \rightarrow+\infty} \frac{1}{n} \sum_{m=1}^{n} r_m \middle| x_{n_3} \right) \right) \\
 & \geq \m{E}^{x_0}_{\sigma_0} \left( [v(x_{n_3})-2\eta(d(x_{n_3},\psi(x_{n_3})))- \varepsilon]1_{A}\right) \\
&\geq v(x_0)- 5 \varepsilon-2\eta(3\varepsilon).
\end{align*}
Because $\eta(0)=0$ and $\eta$ is continuous at $0$, the gambling house $\Gamma(x_0)$ has a pathwise uniform value, and Theorem \ref{Dtheo1} is proved. 
\end{proof}



\section{Proofs of Theorem \ref{D1Lips}, Theorem \ref{DthmMDP} and Theorem \ref{DthmPOMDP}} \label{DproofPOMDP}

This section is dedicated to the proofs of Theorem \ref{D1Lips}, Theorem \ref{DthmMDP} and Theorem \ref{DthmPOMDP}. Theorem \ref{D1Lips} and Theorem \ref{DthmMDP} stem from Theorem \ref{Dtheo1}. Theorem \ref{DthmPOMDP} is not a corollary of Theorem \ref{Dtheo1}. Indeed, applying Theorem \ref{Dtheo1} to the framework POMDPs, would only yield the existence of the uniform value in pure strategies and not the existence of the pathwise uniform value.



\subsection{Proof of Theorem \ref{D1Lips}}
Let $\Gamma:=(X,F,r)$ be a gambling house such that $F$ is 1-Lipschitz. Without loss of generality, we can assume that $r$ is 1-Lipschitz. Indeed, any continuous payoff function can be uniformly approximated by Lipschitz payoff functions, and dividing the payoff function by a constant does not change the decision problem.

In order to prove Theorem \ref{D1Lips}, 
it is sufficient to prove that for all $n \geq 1$, $v_n$ is 1-Lipschitz, and $w_{\infty}$ is 1-Lipschitz. Indeed, it implies that the family $\{v_n,n \geq 1\}$ is uniformly equicontinuous and $w_{\infty}$ is continuous. Theorem \ref{D1Lips} then stems from Theorem \ref{Dtheo1}.

Recall that $G:X \rightrightarrows \Delta(X)$ is defined for all $x \in X$ by $G(x):=\sco F(x)$. 
{\color{black}
\begin{lemma}\label{DG_lipschitz}
The correspondence $G$ is $1$-Lipschitz.
\end{lemma}

%


\begin{proof}
Let $x$ and $x'$ be two states in $X$. Fix $\mu\in G(x)$. Let us show that there exists $\mu' \in G(x')$ such that $d_{KR}(\mu,\mu') \leq  d(x,x')$.\\

\noindent By definition of $G(x)$, there exists $\nu \in \Delta(F(x))$ such that
for all $g\in \mathcal{C}(X,[0,1])$, 
\begin{equation*}
\hat{g}(\mu)=\int_{\Delta(X)} \hat{g}(z) \nu(dz).
\end{equation*}
\noindent Let $M=F(x) \subset \Delta(X)$. We consider the correspondence $\Phi:  M \rightrightarrows \Delta(X)$ defined for $z \in M$ by  
\begin{equation*}
\Phi(z):=\{z' \in F(x') \ | \ d_{KR}(z,z')\leq  d(x,x')\}.
\end{equation*}
Because $F$ is $1$-Lipschitz, $\Phi$ has nonempty values. Moreover, $\Phi$ is the intersection of two correspondences with a closed graph, therefore it is a correspondence with a closed graph. Applying Proposition \ref{Dselection}, we deduce that $\Phi$ has a measurable selector $\phi: M \rightarrow \Delta(X)$.


%

Let $\nu' \in \Delta(\Delta(X))$ be the image measure of $\nu$ by $\phi$. Throughout the paper, we use the following notation for image measures:
\begin{equation*}
\nu':=\nu \circ \phi^{-1}.
\end{equation*}
By construction, $\nu'(F(x'))=1$ and for all $h\in \mathcal{C}(\Delta(X),[0,1])$, 
\[
\int_{\Delta(X)} h(\phi(z))\nu(dz)=\int_{\Delta(X)} h(u)  \nu'(du).
\]
Let $\mu':=\bary(\nu')$ and $f\in E_1$. The function $\hat{f}$ is $1$-Lipschitz, and 
\begin{align*}
\left|\hat{f}(\mu)-\hat{f}(\mu') \right| & = \left| \int_{\Delta(X)} \hat{f}(z)\nu(dz)- \int_{\Delta(X)} \hat{f}(u)\nu'(du) \right| \\
 &  = \left| \int_{\Delta(X)} \hat{f}(z)\nu(dz)- \int_{\Delta(X)} \hat{f}(\phi(z))\nu(dz) \right| \\
 &  \leq \int_{\Delta(X)} \left|  \hat{f}(z)-  \hat{f}(\phi(z)) \right| \nu(dz) \\
& \leq  d(x,x').
\end{align*}

\end{proof}
Because $G$ is $1$-Lipschitz, given $(x,u) \in \graph G$ and $y \in X$, there exists $w \in G(y)$ such that $d_{KR}(u,w) \leq d(x,y)$. For our purpose, we need that the optimal coupling between $u$ and $w$ can be selected in a measurable way. This is the aim of the following lemma:
\begin{lemma}\label{Dselec_couplage}
There exists a measurable mapping $\psi: \graph G \times X \rightarrow \Delta(X \times X)$ such that for all $(x,u) \in \graph G$, for all $y\in X$,
\begin{itemize}
\item the first marginal of  $\psi(x,u,y)$  \text{ is } $u$,
\item the second marginal of $\psi(x,u,y)$ \text{ is in } $G(y)$,
\item $\displaystyle \int_{X\times X} d(s,t)\psi(x,u,y)(ds,dt) \leq d(x,y)$.
\end{itemize}
\end{lemma}

\begin{proof}
Let $S:= \graph (G) \times X$, $X':=\Delta(X \times X)$ and $\Xi:S \rightrightarrows X'$ the correspondence defined for all $(x,u,y) \in S$ by
\begin{equation*}
\Xi(x,u,y)=\left\{\pi \in \Delta(X \times X) \ | \ \pi_1=u, \pi_2 \in G(y) \right\},
\end{equation*}
where $\pi_1$ (resp. $\pi_2$) denotes the first (resp. second) marginal of $\pi$. 
The correspondence $\Xi$ has a closed graph. Let $f: X' \rightarrow \m{R}$ defined by
\begin{equation*}
f(\pi):=\int_{X\times X} d(s,t) \pi(ds,dt).
\end{equation*}
The function $f$ is continuous. Applying the measurable maximum theorem (see \cite[Theorem 18.19, p.605]{hitchhiker}), 
we obtain that the correspondence $\displaystyle s \rightarrow \argmin_{\pi \in \Xi(s)} f(\pi)$ has a measurable selector, which proves the lemma.
\end{proof}

\begin{proposition} \label{Djunction}
Let $x,y \in X$ and $\sigma$ be a strategy. Then there exist a probability measure $\m{P}^{x,y}_{\sigma}$ on $H_{\infty} \times H_\infty$, and a strategy $\tau$ such that:
\begin{itemize}
\item
$\m{P}^{x,y}_{\sigma}$ has first marginal $\m{P}^{x}_{\sigma}$,
\item $\m{P}^{x,y}_{\sigma}$ has second marginal $\m{P}^{y}_{\tau}$,
\item {\color{black} The following inequalities holds: for every $n\geq 1$
\begin{equation*}
\m{E}^{x,y}_{\sigma}\left( \frac{1}{n} \sum_{m=1}^{n}\left|r(X_m)-r(Y_m)\right|\right) \leq d(x,y),
\end{equation*}
and}
\begin{equation*}
\m{E}^{x,y}_{\sigma}\left(\limsup_{n \rightarrow+\infty} \frac{1}{n} \sum_{m=1}^{n}\left|r(X_m)-r(Y_m)\right|\right) \leq d(x,y),
\end{equation*}
where $X_m$ (resp. $Y_m$) is the $m$-th coordinate of the first (resp. second) infinite history.
\end{itemize}
\end{proposition}
\begin{proof}


Define the stochastic process $(X_m,Y_m)_{m \geq 0}$ on $(X \times X)^\N$ such that the conditional distribution of $(X_m,Y_m)$ knowing $(X_l,Y_l)_{0 \leq l\leq m-1}$ is 
\[
\psi(X_{m-1},\sigma(X_0,...,X_{m-1}),Y_{m-1}),
\]
with $\psi$ defined as in Lemma \ref{Dselec_couplage}. Let $\m{P}_{\sigma}^{x,y}$ be the law on $H_{\infty}^2$ induced by this stochastic process and the initial distribution $\delta_{(x,y)}$. By construction, the first marginal of $\m{P}_{\sigma}^{x,y}$ is $\m{P}_{\sigma}^{x}$. \\

For $m \in \m{N}^*$ and $(y_0,...,y_{m-1}) \in X^{m}$, define $\tau_m(y_0,...,y_{m-1}) \in \Delta(X)$ as being the law of $Y_{m}$, conditional to $Y_{0}=y_0,...,Y_{m-1}=y_{m-1}$. By convexity of $G$, this defines a (behavior) strategy $\tau$ in the game $\Gamma$. Moreover, the probability measure $\m{P}_{\tau}^{y}$ is equal to the second marginal of $\m{P}_{\sigma}^{x,y}$. 
\\


{ \color{black}
 
For all $m \in \m{N}^*$, we have $\m{P}_{\sigma}^{x,y}$-almost surely
\begin{eqnarray*}
\m{E}_\sigma^{x,y} \left(d(X_m,Y_m)|X_{m-1},Y_{m-1} \right) & = & \int_{X\times X} d(s',t')\psi(X_{m-1},\sigma(X_0,...,X_{m-1}),Y_{m-1})(ds',dt'),\\
&\leq&d(X_{m-1},Y_{m-1}).
\end{eqnarray*}
}

{\color{black}

\noindent The random process $(d(X_m,Y_{m}))_{m \geq 0}$ is a positive supermartingale. Therefore, we have 
\begin{align*}
\m{E}^{x,y}_{\sigma}\left( \frac{1}{n} \sum_{m=1}^{n}\left|r(X_m)-r(Y_m)\right|\right) & \leq \m{E}^{x,y}_{\sigma}\left( \frac{1}{n} \sum_{m=1}^{n} d(X_m,Y_m) \right), \\
& =   \frac{1}{n} \sum_{m=1}^{n} \m{E}^{x,y}_{\sigma}\left(d(X_m,Y_m) \right), \\
& \leq  d(x,y).
\end{align*}
}
\noindent Moreover, the random process $(d(X_m,Y_{m}))_{m \geq 0}$ converges $\m{P}^{x,y}_{\sigma}$-almost surely to a random variable $D$, such that $\m{E}^{x,y}_{\sigma}(D) \leq d(x,y)$. For every $n\geq 1$, we have
\begin{eqnarray*}
\frac{1}{n} \sum_{m=1}^{n} \left|r(X_m)-r(Y_m)\right| &\leq& 
\frac{1}{n} \sum_{m=1}^{n} d(X_m,Y_m)
\end{eqnarray*}
and the Ces\`aro theorem yields
\begin{equation*}
\limsup_{n \rightarrow +\infty} \frac{1}{n} \sum_{m=1}^{n} \left|r(X_m)-r(Y_m)\right| \leq D \quad \m{P}^{x,y}_{\sigma} \ \text{a.s.}
\end{equation*}  \\
Integrating the last inequality yields the proposition. 
%
\end{proof}

}

Proposition \ref{Djunction} implies that for all $n \geq 1$, $v_n$ is 1-Lipschitz, and that $w_{\infty}$ is 1-Lipschitz. Thus, Theorem \ref{D1Lips} holds.



%
%


\subsection{Proof of Theorem \ref{DthmMDP} for MDPs}

In this subsection, we consider a MDP $\Gamma=(K,I,g,q)$, as described in Subsection \ref{subMDP}: the state space $(K,d_K)$ and the action set $(I,d_I)$ are compact metric, and the transition function $q$ and the payoff function $g$ are continuous. As in the previous section, without loss of generality we assume that the payoff function $g$ is in fact $1$-Lipschitz.
\\

In the model of gambling house, there is no explicit set of actions. In order to apply Theorem \ref{Dtheo1} to $\Gamma$, we put the action played in the state variable. Indeed, we consider an auxiliary gambling house $\widetilde{\Gamma}$, with state space $K \times I \times K$. At each stage $m \geq 1$, the state $x_m$ in the gambling house corresponds to the state $(k_{m},i_{m},k_{m+1})$ in the MDP. Formally, $\widetilde{\Gamma}$ is defined as follows: 
\begin{itemize}
\item
The state space is $X:=K \times I \times K$, equipped with the distance $d$ defined by
\[
\forall (k,i,l), (k',i',l')\in X, \ d((k,i,l),(k',i',l'))= \max(d_K(k,k'),d_I(i,i'),d_K(l,l')).
\]
\item
The payoff function $r:X \rightarrow [0,1]$ is defined by: for all $(k,i,k') \in X$, $r(k,i,k' ):=g(k,i)$.
\item
The correspondence $F: X \rightarrow \Delta(X)$ is defined by:
\begin{equation*} 
\forall (k,i,k') \in K \times I \times K, \ F(k,i,k'):=\left\{\delta_{k',i'}\otimes q(k',i'): i' \in I \right\},
\end{equation*}
where $\delta_{k',i'}$ is the Dirac measure at $(k',i')$, and the symbol $\otimes$ stands for product measure.
\end{itemize}
Fix some arbitrary state $k_0 \in K$ and some arbitrary action $i_0 \in I$. Given an initial state $k_1$ in the MDP $\Gamma$, the corresponding initial state $x_0$ in the gambling house $\widetilde{\Gamma}$ is $(k_0,i_0,k_1)$. 
By construction, the payoff at stage $m$ in $\widetilde{\Gamma}(x_0)$ corresponds to the payoff at stage $m$ in $\Gamma(k_1)$.

Now let us check the assumptions of Theorem \ref{Dtheo1}. The state space $X$ is compact metric. Because $g$ is continuous, $r$ is continuous, and the following lemma holds:

\begin{lemma}\label{DlipMDP}
The correspondence $F$ has a closed graph. 
\end{lemma}

\begin{proof}
Let $(x_n,u_n)_{n\in \N} \in (Graph\ F)^{\N}$ be a convergent sequence. By definition of $F$, for every $n \geq 1$, there exist $(k_n,i_n,k'_n) \in K\times I \times K$ and $i_n' \in I$ such that 
\[
x_n=(k_n,i_n,k'_n),
\]
and
\[
u_n=\delta_{k'_n,i'_n} \otimes q(k'_n,i'_n) .
\]
Moreover, the sequence $(k_n,i_n,k'_n,i'_n)_{n\geq 1}$ converges to some $(k,i,k',i')\in K\times I \times K\times I.$ Because the transition $q$ is jointly continuous, we obtain that $(u_n)$ converges to $\delta_{(k',i')}\otimes q(k',i')$, which is indeed in $F(k,i,k')$.
\end{proof}

We now prove that for all $n \in \m{N}^*$, $v_n$ is 1-Lipschitz, and that $w_{\infty}$ is 1-Lipschitz.  It is more convenient to prove this result in the MDP $\Gamma$, rather than in the gambling house $\widetilde{\Gamma}$. Thus, in the next proposition, $H_\infty=(K\times I)^\infty$ is the infinite history in $\Gamma$, a strategy $\sigma$ is a map from $\cup_{m \geq 1}  K \times  (I \times K)^{m-1}$ to $\Delta(I)$, and $\P^{k_1}_\sigma$ denotes the probability over $H_\infty$ generated by the pair $(k_1,\sigma)$. This proposition is similar to Proposition \ref{Djunction}.

\begin{proposition} \label{Djunction_MDP}
Let $k_1,k_1' \in K$ and $\sigma$ be a strategy. Then there exist a probability measure $\m{P}^{k_1,k_1'}_{\sigma}$ on $H_{\infty} \times H_\infty$, and a strategy $\tau$ such that:
\begin{itemize}
\item
$\m{P}^{k_1,k_1'}_{\sigma}$ has first marginal $\m{P}^{k_1}_{\sigma}$,
\item
$\m{P}^{k_1,k_1'}_{\sigma}$ has second marginal $\m{P}^{k_1'}_{\tau}$,
\item
The following inequalities hold: 
for every $n\geq 1$,
\begin{equation*}
\m{E}^{k_1,k'_1}_{\sigma}\left( \frac{1}{n} \sum_{m=1}^{n}\left|g(K_m,I_m)-g(K'_m,I'_m)\right|\right) \leq d_K(k_1,k'_1),
\end{equation*}
and
\begin{equation*}
\m{E}^{k_1,k'_1}_{\sigma}\left(\limsup_{n \rightarrow+\infty} \frac{1}{n} \sum_{m=1}^{n}\left|g(K_m,I_m)-g(K'_m,I'_m)\right|\right) \leq d_K(k_1,k'_1),
\end{equation*}
where $K_m,I_m$ (resp. $K'_m,I'_m$) is the $m$-th coordinate of the first (resp. second) infinite history.
\item
Under $\m{P}^{k_1,k_1'}_{\sigma}$, for all $m \geq 1$, $I_m=I_m'$. 
\end{itemize}
\end{proposition}
\begin{proof}

Exactly as in Lemma \ref{Dselec_couplage}, one can construct a measurable mapping $\psi:K \times K \times I \rightarrow \Delta(K \times K)$ 
such that for all $(k,k',i) \in K \times K \times I$, $\psi(k,k',i) \in \Delta(K \times K)$ is an optimal coupling between $q(k,i)$ and $q(k',i)$ for the KR distance.\\

We define a stochastic process on $I \times K \times I \times K$, in the following way: given an arbitrary action $i_0$, we set $I_0=I_0'=i_0$, $K_1=k_1$, $K_1'=k'_1$.
Then, for all $m \geq 2$, given $(I_{m-1},K_m,I'_{m-1},K_m')$, we construct $(I_{m},K_{m+1},I'_{m},K'_{m+1})$ as follows:
\begin{itemize}
\item
$I_m$ is drawn from $\sigma(K_1,I_1,...,K_m)$,
\item
$(K_{m+1},K'_{m+1})$ is drawn from $\psi(K_m,K'_m,I_m)$,
\item
we set $I'_m:=I_m$.
\end{itemize}
By construction, $\m{P}^{k_1,k_1'}_{\sigma}$ has first marginal $\m{P}^{k_1}_{\sigma}$. For $m \geq 1$ and $h_m=(k'_1,i'_1,...,k'_m) \in H_m$, define $\tau(h_m) \in \Delta(I)$ as being the law of $I'_m$, conditional to $K'_1=k'_1,I'_1=i_1',...,K'_m=k'_m$. This defines a strategy. 
Moreover, for all $m \geq 1$, we have
\begin{equation*}
\m{E}^{k_1,k_1'}_{\sigma}(d_K(K_{m+1},K'_{m+1})|K_{m},K'_{m}) \leq d_K(K_{m},K'_{m}).
\end{equation*}
The process $(d_K(K_m,K_m'))_{m \geq 1}$ is a positive supermartingale, thus it converges almost surely. We conclude exactly as in the proof of Proposition \ref{Djunction}.
\end{proof}

The previous proposition implies that the value functions $v_n$ and $w_\infty$ are $1$-Lipschitz. Therefore, the family $\{v_n, n \geq1\}$ is equicontinuous, and $w_\infty$ is continuous. By Theorem \ref{Dtheo1}, the gambling house $\widetilde{\Gamma}$ has a pathwise uniform value in pure strategies. It follows that the MDP $\Gamma$ has a pathwise uniform value in pure strategies, and Theorem \ref{DthmMDP} holds.


\begin{remark}
Renault and Venel \cite{RV12} define slightly differently the auxiliary gambling house associated to a MDP. Instead of taking $K \times I \times K$ as the auxiliary state space, they take $[0,1] \times K$, where the first component represents the stage payoff. In our framework, applying this method would lead to a measurability problem, when trying to transform a strategy in the auxiliary gambling house into a strategy in the MDP.
\end{remark}

\subsection{Proof of Theorem \ref{DthmPOMDP} for POMDPs}
In this subsection, we consider a POMDP $\Gamma=(K,I,S,g,q)$, as described in Subsection \ref{subPOMDP}: the state space $K$ and the signal space $S$ are finite, the action set $(I,d_I)$ is compact metric, and the transition function $q$ and the payoff function $g$ are continuous.\\

A standard way to analyze $\Gamma$ is to consider the belief $p_m \in \Delta(K)$ at stage $m$ about the state as a new state variable, and thus consider an auxiliary problem in which the state is perfectly observed and lies in $\Delta(K)$ (see \cite{rhenius74}, \cite{sawaragi70}, \cite{yushkevich76}). The function $g$ is linearly extended to $\Delta(K) \times \Delta(I)$, in the following way: for all $(p,u) \in \Delta(K) \times \Delta(I)$,
\begin{equation*}
g(p,u):=\sum_{k \in K} \int_{I} g(k,i) u(di).
\end{equation*}
Let $\widetilde{q}:\Delta(K) \times I \rightarrow \Delta(\Delta(K))$ be the transition on the beliefs about the state, induced by $q$: if at some stage of the game, the belief of the decision-maker is $p$, and he plays the action $i$, then his belief about the next state will be distributed according to $\widetilde{q}(p,i)$. We extend linearly the transition $\widetilde{q}$ on $\Delta(K) \times \Delta(I)$, in the following way: for all $f \in \mathcal{C}(\Delta(K),[0,1])$,
\begin{equation*}
\int_{\Delta(K)} f(p) \ [\widetilde{q}(p,u)](dp)=\int_I \int_{\Delta(K)} f(p) \ [\widetilde{q}(p,i)](dp) u(di).
\end{equation*}

We can also define an auxiliary gambling house $\widetilde{\Gamma}$, with state space $[0,1] \times I \times \Delta(K)$: at stage $m$, the auxiliary state $x_m$ corresponds to the triple $(g(p_m,i_m),i_m,p_{m+1})$. Formally, the gambling house $\widetilde{\Gamma}$ is defined as follows: 

\begin{itemize}
\item
State space $X:=[0,1] \times I \times \Delta(K)$: the set $\Delta(K)$ is equipped with the norm 1 $\left\|.\right\|_K$, and the distance $d$ on $X$ is $d:=\max(|.|,d_I,\left\|.\right\|_K)$. 
\item
Payoff function $r:X \rightarrow [0,1]$ such that for all $x=(a,i,p) \in X$, $r(x):=a$. 
\item
Correspondence $F: X \rightarrow \Delta(X)$ defined for all $x=(a,i,p) \in X$ by 
\\
$F(x):=\left\{g(p,i') \otimes \delta_{i'} \otimes \widetilde{q}(p,i'): i' \in I \right\}$.
\end{itemize}
Fix some arbitrary $a_0 \in [0,1]$ and $i_0 \in I$. To each initial belief $p_1 \in \Delta(K)$ in $\Gamma$, we associate an initial state $x_0(p)$ in $\widetilde{\Gamma}$ by:
\begin{equation*}
x_0(p_1):=(a_0,i_0,p_1). 
\end{equation*}

\hspace{5mm}

By construction, the payoff at stage $m$ in the auxiliary gambling house $\widetilde{\Gamma}(x_0(p_1))$ corresponds to the payoff $g(p_m,i_m)$ in the POMDP $\Gamma(p_1)$. In particular, for all $n \in \m{N}^*$, the value of the $n$-stage gambling house $\widetilde{\Gamma}(x(p_1))$ coincides with the value of the $n$-stage POMDP $\Gamma(p_1)$, which is denoted by $v_n(p_1)$. 

One could check that $\widetilde{\Gamma}$ satisfies the assumptions of Theorem \ref{Dtheo1} and therefore has a pathwise uniform value. This would especially imply that $\widetilde{\Gamma}$ has a uniform value in pure strategies, and it would prove that $\Gamma$ has a uniform value in pure strategies. Indeed, let $p_1 \in \Delta(K)$ and $\widetilde{\sigma}$ be a strategy in $\widetilde{\Gamma}(x_0(p_1))$. Let $\sigma$ be the associated strategy in the POMDP $\Gamma(p_1)$. For all $n \geq 1$, we have
\begin{eqnarray*}
\m{E}^{x_0}_{\widetilde{\sigma}}\left( \frac{1}{n} \sum_{m=1}^{n} r(x_m) \right) &=&
\m{E}^{p_1}_{\sigma}\left( \frac{1}{n} \sum_{m=1}^{n} g(p_m,i_m) \right)
\\
&=&
 \m{E}^{p_1}_{\sigma}\left(\frac{1}{n} \sum_{m=1}^{n} g(k_m,i_m) \right).
\end{eqnarray*}
Consequently, the fact that $\widetilde{\Gamma}(x_0(p_1))$ has a uniform value in pure strategies implies that $\Gamma(p_1)$ also has a uniform value in pure strategies.

Unfortunately, this approach does not prove Theorem \ref{DthmPOMDP}, i.e. the existence of the pathwise uniform value in $\Gamma$, due to the following problem: 

\begin{problem} 
It may happen that
\begin{equation*}
\m{E}^{x_0(p_1)}_{\widetilde{\sigma}}\left(\liminf_{n \rightarrow+\infty} \frac{1}{n} \sum_{m=1}^{n} r(x_m) \right)>
\m{E}^{p_1}_{\sigma}\left(\liminf_{n \rightarrow+\infty} \frac{1}{n} \sum_{m=1}^{n} g(k_m,i_m) \right).
\end{equation*}
Indeed, $r(x_m)$ is not equal to $g(k_m,i_m)$: it is the expectation of $g(k_m,i_m)$ with respect to $p_m$. Consequently, the fact that $\widetilde{\sigma}$ is a pathwise $\epsilon$-optimal strategy in $\widetilde{\Gamma}(x_0(p_1))$ does not imply that $\sigma$ is a pathwise $\epsilon$-optimal strategy in $\Gamma(p_1)$.
\end{problem}

To prove Theorem \ref{DthmPOMDP}, we adapt the proof of Theorem \ref{Dtheo1} to the framework of POMDPs. Recall that the proof of Theorem \ref{Dtheo1} was decomposed into three lemmas (Lemmas \ref{Dinvariant}, \ref{Dbirkhoff} and \ref{Djunction_opt}) and a conclusion (Subsection \ref{concl}). We adapt the three lemmas, and the conclusion is similar. \\


 In order to obtain the first lemma, we check that $F$ has a closed graph.

\begin{proposition}
The correspondence $F$ has a closed graph.
\end{proposition}

\begin{proof}
Let $(x_n,u_n)_{n\in \N} \in (Graph\ F)^{\N}$ be a sequence that converges to $(x,u) \in X\times \Delta(X)$. By definition of $F$, for every $n\geq 1$ there exists $(a_n,i_n,p_n,i'_n) \in ([0,1]\times I \times \Delta(K) \times I)$ such that 
\[
x_n=(a_n,i_n,p_n),
\]
and
\[
u_n=g(p_n,i'_n)\otimes \delta_{i'_n} \otimes \widetilde{q}(p_n,i'_n).
\]
It follows that the sequence $(a_n,i_n,p_n,i'_n)_{n\geq 1}$ converges to some $(a,i,p,i')\in [0,1]\times I \times \Delta(K) \times I$ and $x=(a,i,p)$.\\

By Feinberg \cite[Theorem 3.2]{feinberg15}, the function $\widetilde{q}$ is jointly continuous. Because the payoff function $g$ is also continuous, we obtain that $u_n$ converges to $u=q(p,i')\otimes \delta_{i'} \otimes \widetilde{q}(p',i')$ which is indeed in $F(x)$.
\end{proof}


Now we can apply Lemma \ref{Dinvariant} to the gambling house $\widetilde{\Gamma}$. 
For $p \in \Delta(K)$, define $v(p):=\limsup_{n \rightarrow+\infty} v_n(p)$.
Note that for all $x=(a,i,p) \in X$, the set $F(x)$ depends only on the third component $p$. Thus, Lemma \ref{Dinvariant} implies the following lemma for the POMDP $\Gamma$:

\begin{lemma} \label{Dinvariant2}
Let $p_1 \in \Delta(K)$. There exists a distribution $\mu^* \in \Delta(\Delta(K))$ and a stationary strategy $\sigma^*: \Delta(K) \rightarrow \Delta(I)$ such that 
\begin{itemize}
\item $\mu^*$ is $\sigma^*$-invariant: for all $f \in \mathcal{C}(\Delta(K),[0,1])$,
\begin{equation*}
\int_{\Delta(K)} \hat{f}(\widetilde{q}(p,\sigma^*(p)) \mu^*(dp)=\hat{f}(\mu^*)
\end{equation*}
\item For every $\varepsilon>0$ and $N\geq 1$, there exists a (pure) strategy $\sigma$ in $\Gamma$ and $n \geq N$ such that $\sigma$ is $0$-optimal in $\Gamma_n(p_1)$ and 
\[
d_{KR}\left(\frac{1}{n} \sum_{m=1}^n z_m(p_1,\sigma), \mu^* \right) \leq \varepsilon,
\]
where $z_m(p_1,\sigma)$ is the distribution over $\Delta(K)$ at stage $m$, starting from $p_1$,
\item $\displaystyle \int_{\Delta(K)} g(p,\sigma^*(p)) \mu^*(dp)=\hat{v}(\mu^*)=v(p_1)$.
\end{itemize}
\end{lemma}

We can now state a new lemma about pathwise convergence in $\Gamma$.  This replaces Lemma \ref{Dbirkhoff}. 

\begin{lemma}\label{Dbirkhoff2}
Let $p_1 \in \Delta(K)$ and $\mu^*$ be the corresponding measure in the previous lemma. There exists a measurable set $B \subset \Delta(K)$ such that $\mu^*(B)=1$ and for all $p \in B$,
\begin{equation*}
\m{E}^p_{\sigma^*} \left(\liminf_{n \rightarrow+\infty} \frac{1}{n} \sum_{m=1}^{n} g(k_m,i_m) \right) =v(p) \quad \m{P}^{p}_{\sigma^*}-\text{a.s.}
\end{equation*}
\end{lemma}
\begin{proof}
It is not enough to apply Birkhoff's theorem to the Markov chain $(p_m)_{m \geq 1}$, due to the problem mentioned previously.  Instead, we consider the random process $(y_m)_{m \geq 1}$ on $Y:=K \times I \times \Delta(K)$, defined for all $m \geq 1$ by $y_m:=(k_m,i_m,p_m)$: (current state, action played, belief about the current state). Under $\m{P}^{\mu^*}_{\sigma^*}$, 
this is a Markov chain. Indeed, given $m \geq 1$ and $(y_1,y_2,...,y_m) \in Y^m$, the next state $y_{m+1}$ is generated in the following way:
\begin{itemize}
\item
a pair $(k_{m+1},s_m)$ is drawn from $q(k_m,i_m)$,
\item
the decision-maker computes the new belief $p_{m+1}$ according to $p_m$ and $s_m$,
\item
the decision-maker draws an action $i_{m+1}$ from $\sigma^*(p_{m+1})$.
\end{itemize}
By construction, the law of $y_{m+1}$ depends only on $y_m$, and $(y_m)_{m \geq 1}$ is a Markov chain. Define $\nu^* \in \Delta(Y)$ such that the third marginal of $\nu^*$ is $\mu^*$, and for all $p \in \Delta(K)$, the conditional law $\nu^*(.|p) \in \Delta(K \times I)$ is $p \otimes \sigma(p)$. Under $\m{P}^{\mu^*}_{\sigma^*}$, for all $m \geq 1$, the third marginal of $y_m$ is distributed according to $\mu^*$. Moreover, conditional on $p_m$, the random variables $k_m$ and $i_m$ are independent, the conditional distribution of $k_m$ knowing $p_m$ is $p_m$, and the conditional distribution of $i_m$ knowing $p_m$ is $\sigma^*(p_m)$. Thus, $\nu^*$ is an invariant measure for the Markov chain $(y_m)_{m \geq 1}$. Define a measurable map $f:Y \rightarrow [0,1]$ by: for all $(k,i,p) \in Y$, $f(k,i,p)=g(k,i)$. 
Now we can apply Theorem \ref{Dergodic} to $(y_m)_{m \geq 1}$, and deduce that there exist $B_0 \subset \Delta(K)$ and $w: Y \rightarrow [0,1]$ such that for all $p \in B_0$, 
\begin{equation} \label{Deqbirkhoff}
\frac{1}{n} \sum_{m=1}^{n} f(y_m) \underset{n \rightarrow +\infty}{\rightarrow} w(k_1,i_1,p) \quad \m{P}^p_{\sigma^*}-\text{almost  surely,}
\end{equation}
and
\begin{equation*}
\hat{w}(\nu^*)=\hat{f}(\nu^*).
\end{equation*}
By definition of $f$, for all $m \geq 1$, $f(y_m)=g(k_m,i_m)$. 
Moreover, by definition of $\nu^*$, we have
\begin{eqnarray*}
\hat{f}(\nu^*)=\int_{\Delta(K)} g(p,\sigma^*(p)) \mu^*(dp),
\end{eqnarray*}
and by Lemma \ref{Dinvariant2}, we deduce that $\hat{f}(\nu^*)=\hat{v}(\mu^*)$. Consequently, $\hat{w}(\nu^*)=\hat{v}(\mu^*)$. Given $p \in B_0$, denote by $w_0(p)$ the expectation of $w(.,p)$ with respect to $\m{P}^p_{\sigma^*}$. By Equation (\ref{Deqbirkhoff}), we have
\begin{equation*}
\m{E}^p_{\sigma^*} \left(\lim_{n \rightarrow +\infty} \frac{1}{n} \sum_{m=1}^{n} g(k_m,i_m) \right)=w_0(p). 
\end{equation*}
Let us prove that $w_0=v \quad \m{P}^{\mu^*}_{\sigma^*}$-almost surely. Note that $\hat{w_0}(\mu^*)=\hat{w}(\nu^*)=\hat{v}(\mu^*)$. Consequently, it is enough to show that $w_0 \leq v \quad \m{P}^{\mu^*}_{\sigma^*}$-almost surely.
By the dominated convergence theorem and the definition of $v$, we have
\begin{eqnarray*}
\m{E}^p_{\sigma^*} \left(\lim_{n \rightarrow +\infty} \frac{1}{n} \sum_{m=1}^{n} g(k_m,i_m) \right)
&=& 
\lim_{n \rightarrow +\infty} \m{E}^p_{\sigma^*} \left( \frac{1}{n} \sum_{m=1}^{n} g(k_m,i_m) \right)
\\
&\leq& v(p),
\end{eqnarray*}
and the lemma is proved.
\end{proof}


For every $n\geq 1$, the value function $v_n$ is $1$-Lipschitz, as a consequence of the following proposition.  


\begin{proposition}\label{Lip1}
Let $p,p' \in \Delta(K)$ and $\sigma$ be a strategy in $\Gamma$. Then, for every $n\geq 1$, 
\begin{equation*}
\left|\m{E}^p_{\sigma}\left(\frac{1}{n} \sum_{m=1}^n g(k_m,i_m)\right)-\m{E}^{p'}_{\sigma}\left(\frac{1}{n} \sum_{m=1}^n g(k_m,i_m)\right)\right| \leq \|p-p'\|_1.
\end{equation*}
\end{proposition}
This proposition is proved in Rosenberg, Solan and Vieille \cite[Proposition 1]{RSV02}. In their framework, $I$ is finite, but the fact that $I$ is compact does not change the proof at all.

Last, we establish the junction lemma, which replaces Lemma \ref{Djunction_opt}.

\begin{lemma}\label{Djunction_opt2}
Let $p,p' \in \Delta(K)$ and $\sigma$ be a strategy such that
\begin{equation*}
\m{E}^p_{\sigma} \left(\lim_{n \rightarrow+\infty} \frac{1}{n} \sum_{m=1}^{n} g(k_m,i_m)\right)=v(p).
\end{equation*}
Then, the following inequality holds:
\begin{equation*}
\m{E}^{p'}_\sigma\left(\liminf_{n \rightarrow+\infty} \frac{1}{n} \sum_{m=1}^{n} g(k_m,i_m) \right) \geq v(p')-2\left\|p-p'\right\|_1.
\end{equation*}
\end{lemma}
\begin{proof}
Let $k \in K$ and $p_1 \in \Delta(K)$. Denote by $\m{P}^{p_1}_{\sigma}(h_\infty|k)$ the law of the infinite history $h_{\infty} \in (K \times I \times S)^{\m{N}^*}$ in the POMDP $\Gamma(p_1)$, under the strategy $\sigma$, and conditional to $k_1=k$. Then $\m{P}^{p}_{\sigma}(h_\infty|k)=\m{P}^{p'}_{\sigma}(h_\infty|k)$ and 
\begin{equation*}
\m{E}^{p'}_\sigma\left(\liminf_{n \rightarrow+\infty} \frac{1}{n} \sum_{m=1}^{n} g(k_m,i_m) \right) \geq v(p)-\left\|p-p'\right\|_1.
\end{equation*}
For every $n\geq 1$, $v_n$ is $1$-Lipschitz, thus the function $v$ is also $1$-Lipschitz, and the lemma is proved.
\end{proof}

The conclusion of the proof is similar to Section \ref{concl}. Note that apart from the three main lemmas, the only additional property used in Section \ref{concl} was that the family $(v_n)_{n\geq1}$ is uniformly equicontinuous. For every $n\geq 1$, $v_n$ is $1$-Lipschitz, thus the family $(v_n)_{n\geq1}$ is indeed uniformly equicontinuous.

\section*{Acknowledgments}
Both authors gratefully  acknowledge   the support of the Agence Nationale de la Recherche, under grant ANR JEUDY, ANR-10-BLAN 0112, and thank Eugene A. Feinberg, Fabien Gensbittel, J\'er\^ome Renault and Eilon Solan for fruitful discussions.

\bibliography{bibliogen}

\end{document}